\documentclass[11pt, english]{amsart}

\usepackage{amsmath,amssymb,enumerate}

\usepackage[T1]{fontenc}
\usepackage[all]{xy}

\usepackage{babel}
\usepackage{amstext}
\usepackage{amsmath}
\usepackage{amsfonts}
\usepackage{latexsym}
\usepackage{ifthen}

\usepackage{xypic}
\xyoption{all}
\newcommand{\sK}{{\mathcal K}}

\newtheorem{lemma1}{}[section]

\newenvironment{lemma}{\begin{lemma1}{\bf Lemma.}}{\end{lemma1}}

\newenvironment{theorem}{\begin{lemma1}{\bf Theorem.}}{\end{lemma1}}
\newenvironment{proposition}{\begin{lemma1}{\bf Proposition.}}{\end{lemma1}}
\newenvironment{corollary}{\begin{lemma1}{\bf Corollary.}}{\end{lemma1}}
\newenvironment{remark}{\begin{lemma1}{\bf Remark.}\rm}{\end{lemma1}}

\newenvironment{definition}{\begin{lemma1}{\bf Definition.}}{\end{lemma1}}
\newenvironment{notation}{\begin{lemma1}{\bf Notation.}}{\end{lemma1}}

\newenvironment{assumption}{\begin{lemma1}{\bf Assumption.}}{\end{lemma1}}

\newenvironment{remark*}{{\bf Remark.}}{}
\newenvironment{example*}{{\bf Example.}}{}
\newenvironment{assumption*}{{\bf Assumption.}}{}

\newcommand{\R}{\ensuremath{\mathbb{R}}}
\newcommand{\Q}{\ensuremath{\mathbb{Q}}}

\newcommand{\C}{\ensuremath{\mathbb{C}}}
\newcommand{\N}{\ensuremath{\mathbb{N}}}
\newcommand{\PP}{\ensuremath{\mathbb{P}}}

\newcommand{\holom}[3]{\ensuremath{#1:#2  \rightarrow #3}}
\newcommand{\fibre}[2]{\ensuremath{#1^{-1} (#2)}}

\makeatletter
\ifnum\@ptsize=0  \fi \ifnum\@ptsize=2  \fi 

\newcommand\sA{{\mathcal A}}

\newcommand\sH{{\mathcal H}}

\newcommand\sO{{\mathcal O}}

\newcommand\sD{{\mathcal D}}

\DeclareMathOperator*{\pic}{Pic}
\DeclareMathOperator*{\sing}{sing}

\DeclareMathOperator*{\nons}{nons}

\newcommand{\Chow}[1]{\ensuremath{\mbox{\rm Chow}(#1)}}

\newcommand{\NEX}{\overline{\mbox{NE}}(X)}
\newcommand{\NAX}{\overline{\mbox{NA}}(X)}

\newcommand{\NA}[1]{ \ensuremath{ \overline { \mbox{NA} }(#1)} }

\title{Mori fibre spaces for K\"ahler threefolds} 
\date{\today}

\subjclass[2000]{32J27, 14E30, 14J30, 32J17, 32J25}
\keywords{MMP, rational curves, Zariski decomposition, K\"ahler manifolds}

\author{Andreas H\"oring}
\author{Thomas Peternell}

\address{Andreas H\"oring, Laboratoire de Math{\'e}matiques J.A. Dieudonn{\'e},
UMR 7351 CNRS, Universit{\'e} de Nice Sophia-Antipolis, 06108 Nice Cedex 02, France        
}
\email{hoering@unice.fr}

\address{Thomas Peternell, Mathematisches Institut, Universit\"at Bayreuth, 95440 Bayreuth, 
Germany}

\email{thomas.peternell@uni-bayreuth.de}

\begin{document}

\begin{abstract} 
Let $X$ be a compact K\"ahler threefold such that the 
base of the MRC-fibration has dimension two.
We prove that $X$ is bimeromorphic to a Mori fibre space.
Together with our earlier result \cite{HP15} this completes the MMP for compact K\"ahler threefolds: let $X$ be a non-projective compact K\"ahler threefold. Then $X$ has a minimal model or 
$X$ is bimeromorphic to a Mori fibre space over a non-projective K\"ahler surface.
\end{abstract}

\maketitle

\section{Introduction}

This paper continues our study of the minimal model program (MMP) for compact K\"ahler threefolds. In \cite{HP15} we established the existence of minimal models 
for compact K\"ahler threefolds such that $K_X$ is pseudoeffective. 
More precisely, minimal models are obtained,
as in the projective setting, by a sequence of contractions of extremal rays (in a suitable cone) and flips. 
By a theorem of Brunella \cite{Bru06} a smooth compact K\"ahler threefold has pseudoeffective $K_X$ if and only if $X$ is not uniruled. 
In the present work we deal with the remaining case
where $X$ is uniruled. The general fibre of the MRC-fibration $X \dashrightarrow Z$ is rationally connected, so
carries no holomorphic forms \cite[Cor.4.18]{Deb01}. Thus if the base $Z$ has dimension at most one,
then we obtain $H^2(X,\sO_X) = H^0(X, \Omega_X^2) = 0$.
In particular the K\"ahler manifold $X$ is projective by Kodaira's criterion. Since our main interest is the study
of non-projective K\"ahler threefolds, we focus on the
case where $Z$ has dimension two:

\begin{theorem} \label{theoremMFS}
Let $X$ be a normal $\Q$-factorial compact K\"ahler threefold with at most terminal singularities.
Suppose that the base of the MRC-fibration $X \dashrightarrow Z$ has dimension two.

Then $X$ is bimeromorphic to a Mori fibre space, i.e. 
there exists a MMP
$$
X \dashrightarrow X',
$$
consisting of contractions of extremal rays and flips, such that $X'$ admits a fibration $\holom{\varphi}{X'}{S}$ onto a normal compact  $\mathbb Q-$factorial K\"ahler surface with at most
klt singularities such that $-K_{X'}$ is $\varphi$-ample and $\rho(X'/S)=1$.
\end{theorem}

It will be important to work with a special type of K\"ahler classes: 

\begin{definition}
Let $X$ be a normal $\Q$-factorial compact   K\"ahler threefold with at most terminal singularities.
Suppose that the base of the MRC-fibration $X \dashrightarrow Z$ has dimension two, and let $F \simeq \PP^1$ be a general fibre. 
A K\"ahler class $\omega$ on $X$ is normalised if $\omega \cdot F=2$.
\end{definition} 

Since the canonical class $K_X$ has degree $-2$ on $F$, 
the adjoint class $K_X+\omega$ is trivial on $F$. 
Using a recent result of P\v aun \cite{Pau12} we first prove
that $K_X+\omega$ is pseudoeffective. 
The proof of Theorem \ref{theoremMFS} then proceeds in two steps, the first being the existence of a MMP for
the adjoint class $K_X+\omega$:

\begin{theorem} \label{theoremadjointMMP}
Let $X$ be a normal $\Q$-factorial compact K\"ahler threefold with at most terminal singularities.
Suppose that the base of the MRC-fibration $X \dashrightarrow Z$ has dimension two.
Then there exists a MMP 
$$
X \dashrightarrow X'
$$
such that for every normalised K\"ahler class $\omega'$ on $X'$ 
the adjoint class $K_{X'}+\omega'$ is nef.
\end{theorem}

Once we have a normalised K\"ahler class $\omega$ such that
$K_X+\omega$ is nef, the adjoint class $K_X + \omega$ is a natural candidate for the ``nef supporting class'' that defines a Mori fibre space structure. \\
The second step is to prove an analogue of the base-point free theorem for the adjoint class $K_X+\omega$.

\begin{theorem} \label{theorembpf}
Let $X$ be a normal $\Q$-factorial compact   K\"ahler threefold with at most terminal singularities.
Suppose that the base of the MRC-fibration $X \dashrightarrow Z$ has dimension two.
Let $\omega$ be a normalised K\"ahler class on $X$ such that $K_X+ \omega$ is nef.

Then there exists a {\em holomorphic} fibration $\holom{\varphi}{X}{S}$ onto a normal compact K\"ahler  surface $S$ such that $K_X+\omega$ is $\varphi$-trivial. 
\end{theorem}

By construction, the anticanonical class $-K_X$ is ample with respect to the fibration $X \rightarrow S$,
so we can use the cone and contraction theorem for projective morphisms (\cite{Nak87}, \cite{KM98}) to run a relative MMP. 
This MMP terminates with the Mori fibre space we are looking for.

In the situation of Theorem \ref{theorembpf} 
one can prove that $S$ is $\Q$-factorial with at most rational singularities, but
it is not quite clear whether $S$ is klt. However we can
prove this property for an elementary contraction
of fibre type, cf. Lemma \ref{lemmacontractionfibretype}.

{\bf Acknowledgements.} 
We thank the Forschergruppe 790 ``Classification of algebraic surfaces and compact complex manifolds'' of the Deutsche
Forschungsgemeinschaft for
financial support. 
A. H\"oring was partially also supported by the A.N.R. project CLASS\footnote{ANR-10-JCJC-0111}.

\section{Notation}

We use the same notation as in \cite{HP15}.
For the convenience of the reader we recall the most 
important definitions and basic results. 

\begin{definition} \label{definitionkaehler}
An irreducible and reduced complex space $X$ is K\"ahler if there exists a K\"ahler form $\omega$, i.e. a positive closed real $(1,1)$-form $\omega \in \mathcal A_\R^{1,1}(X)$, 
such that the following holds: for every point
$x \in X_{\sing}$ there exists an open neighbourhood $x \in U \subset X$ and a closed embedding $i_U: U \subset V$ into an open set $V \subset \mathbb C^N$, and  
a strictly plurisubharmonic $C^{\infty}$-function $f : V \rightarrow \C$ with 
$ \omega|_{U \cap X_{\nons}} = (i \partial \overline \partial f)|_{U \cap X_{\nons}}$. 
\end{definition} 

In the same manner one can define $(p,q)-$forms on an irreducible reduced complex space \cite{Dem85}, by duality we obtain the usual notions of currents. 

We will next define the appropriate analogue of the N\'eron-Severi
space $N^1(X)$ for a normal compact K\"ahler space, as well as 
the cones $\NEX$ and $\NAX$ contained in its dual $N_1(X)$. 
For any details we refer to \cite[Sect.3]{HP15}.

\begin{definition}  \cite[Defn. 4.6.2]{BEG13}  \cite[Defn.3.6]{HP15}
Let $X$ be an irreducible reduced complex space. Let $\sH_X$ be the sheaf of real parts of holomorphic
functions multiplied with $i$. A $(1,1)$-form with local potentials on $X$ is a global section
of the quotient sheaf $\sA^0_X/\sH_X$. We define the Bott-Chern cohomology
$$
N^1(X) := H^1(X, \sH_X).
$$
\end{definition}

\begin{remark} \label{bott} {\rm 
Using the exact sequence 
$$
0 \rightarrow \sH_X \rightarrow \sA^0_X \rightarrow \sA^0_X/\sH_X \rightarrow 0,
$$
and the fact that $\sA^0_X$ is acyclic, we obtain a surjective map
$$
H^0(X,  \sA^0_X/\sH_X) \rightarrow H^1(X, \sH_X).
$$
Thus we can see an element of the Bott-Chern cohomology group as a closed $(1,1)$-form with local potentials
modulo all the forms that are globally of the form $dd^c u$.  \\
Let $\sD_X$ be the sheaf of distributions. Using the exact sequence 
$$
0 \rightarrow \sH_X \rightarrow \sD_X \rightarrow \sD_X/\sH_X \rightarrow 0,
$$
we see that one obtains the same Bott-Chern group, considering $(1,1)-$currents $T$ with local potentials, which is to say that 
locally $T = dd^cu$ with $u$ a distribution.

}
\end{remark}

Dually we define

\begin{definition} Let $X$ be a normal compact complex space. Then $N_1(X)$ is the vector space of real closed currents of bidimension $(1,1)$ modulo the following equivalence relation:
$T_1 \sim T_2 $ if and only if $$T_1(\eta) = T_2(\eta)$$ for all real closed $(1,1)$-forms $\eta$.
\end{definition} 

In \cite[Prop.3.9]{HP15} we established a canonical isomorphism 
\begin{equation} \label{caniso} \Phi: N^1(X ) \to N_1(X)^* \end{equation} 
for any normal compact complex space $X$ in the Fujiki class $\mathcal C,$ i.e., for those $X$ which are bimeromorphic to a K\"ahler space.

\begin{definition} \label{definitionNA}
Let $X$ be a normal compact complex space in class $\mathcal C$. 
We define $\overline {NA}(X) \subset N_1(X)$ as the cone generated by the positive closed currents of bidimension $(1,1)$. 
\end{definition}

Given an irreducible curve $C \subset X$, we associate to $C$ the current of integration $T_C$. In the case  of isolated singularities, which is the  only case
relevant in our setting, we define

$$T_C(\omega) = \int_C \omega = \int_{\hat C} \pi^*(\omega), $$
where $\pi: \hat X \to X$ is a resolution of singularities, the curve $\hat C $ is the strict transform of $C$, and $\omega$ a
$d$-closed $(1,1)$-form on $X$. 
We define  the Mori cone $\overline {NE}(X) \subset N_1(X)$ as the closure of the cone generated by the currents $T_C$
and clearly have an inclusion
$$
\overline {NE}(X) \subset \overline {NA}(X). 
$$

\begin{definition} \label{definitionnefcone}
Let $X$ be an irreducible reduced compact complex space in class $\mathcal C$. 
We denote by ${\rm Nef}(X) \subset N^1(X)$ the cone generated by cohomology classes which are
nef in the sense of \cite[Defn.3]{Pau98}:
let $u \in N^1(X)$ be 
a class represented by a form $\alpha$ with local potentials. 
 Then $u$ is nef if for some positive $(1,1)-$form $\omega $ on $X$ and
for every $\epsilon > 0$ there exists $f_{\epsilon} \in \sA^0(X)$ such that
$$ \alpha + i \partial \overline{\partial} f_{\epsilon} \geq - \epsilon \omega.$$
The class $u$ is pseudo-effective, if it can be represented by a current $T$ which is locally of the form $T =   \partial \overline{\partial} \varphi$ with
$\varphi$ a plurisubharmonic function. 
\end{definition}

If $X$ is a normal compact K\"ahler space, we can also consider the open cone $\sK$ generated by
the classes of K\"ahler forms. In this case we know by \cite[Prop.6.1]{Dem92}\footnote{The statement in \cite[Prop.6.1.iii)]{Dem92} is for compact manifolds,
but the proof works in the singular setting, cf. also \cite[Rem.3.5]{HP15} for regularisation arguments in the singular setting.} that 
$$
{\rm Nef}(X)  = \overline {\sK}.
$$
 
 As in the projective setting, we have a duality statement:
 
\begin{proposition} \label{dual1} \cite[Prop.3.15]{HP15}
Let $X$ be a normal compact threefold in class $\mathcal C$. 
Then the cones ${\rm Nef}(X)$ and $ \overline{NA}(X)$ are dual via the canonical 
isomorphism $\Phi: N^1(X) \to N_1(X)^*$ given by (\ref{caniso}).
\end{proposition}

Finally we define the notion of the contraction of an extremal ray $R$. It is very important 
to consider extremal rays in the dual K\"ahler cone $\NAX$ rather than
in the Mori cone $\NEX$. 

\begin{definition} \label{definitioncontraction}
Let $X$ be a normal $\Q$-factorial compact K\"ahler space with at most terminal  singularities, and let $\omega$ be a K\"ahler class on $X$.
Let $R$ be a $(K_X + \omega)$-negative extremal ray in $\NAX$.
A contraction of the extremal ray $R$ is
a morphism \holom{\varphi}{X}{Y} onto a normal compact K\"ahler space $Y$,
such that $-(K_X+\omega)$ is a K\"ahler class on every fibre and
a curve $C \subset X$ is contracted if and only if $[C] \in R$.
\end{definition}

\section{MMP for the adjoint class}

In order to simplify the statements we will work under the following

\begin{assumption} \label{assumption}
Let $X$ be a normal $\Q$-factorial compact  K\"ahler threefold with at most terminal singularities.
Suppose that the base of the MRC-fibration $X \dashrightarrow Z$ has dimension two,
and let $\omega$ be a normalised K\"ahler class on $X$.
\end{assumption}

\begin{remark} \label{remarkassumption}
Observe that the surface $Z$ is not uniruled: since this is a bimeromorphic statement we can suppose that $X$ and $Z$ are smooth and the MRC-fibration is a morphism $\varphi: X \rightarrow Z$.
If $(C_t)_{t \in T} \subset Z$ is a dominant family of rational curves, the
surface $\fibre{\varphi}{C_t}$ is uniruled by the fibres of $\fibre{\varphi}{C_t} \rightarrow C_t$. Thus it carries no holomorphic $2$-form, in particular $\fibre{\varphi}{C_t}$ is projective by Kodaira' s criterion.
Thus Tsen's theorem applies and we obtain that $\fibre{\varphi}{C_t}$ is rationally connected. Now we conclude as in the algebraic case that $\varphi$ is not the MRC-fibration.
The same line of arguments also shows that the theorem of Graber-Harris-Starr \cite{GHS03} is also true in K\"ahler category.
\end{remark}

\subsection{Remarks on adjunction} \label{subsectionsurfaces}
Let $X$ be a normal  $\Q$-factorial compact K\"ahler threefold with at most terminal  singularities. 
Let $S \subset X$ be a prime divisor, i.e. an irreducible and reduced compact surface. Let $m \in \N$ be the smallest positive integer such that
both $mK_X$ and $mS$ are Cartier divisors on $X$. Then the canonical class $K_S \in \pic(S) \otimes \Q$ is defined by 
$$
K_S := \frac{1}{m} (mK_X+mS)|_S.
$$
Since $X$ is smooth in codimension two, there
exist at most finitely many points $\{p_1, \ldots p_q \}$ where $K_X$ and $S$ are not Cartier. 
Thus by the adjunction formula $K_S$ is isomorphic to the dualising sheaf $\omega_S$ on $S \setminus  \{p_1, \ldots p_q \}$.

Let now \holom{\nu}{\tilde S}{S} be the normalisation. Then we have
\begin{equation} \label{equationconductor}
K_{\tilde S} \sim_\Q \nu^* K_S - N,
\end{equation}
where $N$ is an effective Weil divisor defined by the conductor ideal. Indeed this formula holds 
by \cite{Rei94} for the dualising sheaves. Since $\sO_{\tilde S}(\nu^* K_S)$ is isomorphic to $\nu^* \omega_S$ on the complement
of $\fibre{\nu}{p_1, \ldots p_q}$, the formula holds for the canonical classes.

Let \holom{\mu}{\hat S}{\tilde S} be the minimal resolution of the normal surface $\tilde S$, then we have
$$
K_{\hat S} \sim_\Q \mu^* K_{\tilde S} - N',
$$
where $N'$ is an effective $\mu$-exceptional $\Q$-divisor \cite[4.1]{Sak81}. Thus if \holom{\pi}{\hat S}{S} is the composition $\nu \circ \mu$, there exists an
effective, canonically defined $\Q$-divisor $E \subset \hat S$ such that
\begin{equation} \label{equationresolve}
K_{\hat S} \sim_\Q \pi^* K_S - E.
\end{equation}
Let $C \subset S$ be a curve such that $C \not\subset S_{\sing}$. Then the morphism $\pi$ is an isomorphism at the general point of $C$,
and we can define the strict transform $\hat C \subset \hat S$ as the closure of $C \setminus S_{\sing}$. Since $\hat C$ is an (irreducible) curve that is not contained in 
the divisor $N$ defined by the conductor, we have $\hat C \not\subset E$. By the projection formula
and \eqref{equationresolve} we obtain
\begin{equation} \label{inequalitycanonical}
K_{\hat S} \cdot \hat C \leq K_S \cdot C.
\end{equation}

\subsection{Divisorial Zariski decomposition for $K_X+\omega$} \label{subsectiondivisorial}
The starting point of our investigation is the following observation:

\begin{lemma} \label{lemmapseudoeffective}
Under the Assumption \ref{assumption} the adjoint class $K_X+\omega$ is pseudoeffective.
\end{lemma}

\begin{proof}
Being pseudoeffective is a closed property in $N^1(X)$, 
so it is sufficient to prove that for every $\varepsilon>0$,
the class $K_X+ (1+\varepsilon) \omega$ is pseudoeffective. 
Let $\holom{\mu}{X'}{X}$ be a bimeromorphic morphism from a smooth K\"ahler threefold $X'$
such that the MRC-fibration is a morphism $\varphi': X' \rightarrow Z'$ onto a smooth surface $Z'$. 
The projection formula yields
$$
\mu_* \left(K_{X'}+ (1+\varepsilon) \mu^*\omega\right) = K_X+ (1+\varepsilon) \omega,
$$
so it is sufficient to prove that $K_{X'}+ (1+\varepsilon) \mu^* \omega$ is pseudoeffective. However by a recent result of
P\v aun \cite[Thm.1.1]{Pau12}, the class $K_{X'/Z'}+ (1+\varepsilon) \mu^* \omega$ is pseudoeffective. 
Since the surface $Z'$ is not uniruled (cf. Remark \ref{remarkassumption}) and K\"ahler by \cite[Thm.3]{Var86},
the canonical class $K_{Z'}$ is pseudoeffective. Thus $K_{X'}+ (1+\varepsilon) \mu^* \omega$ is pseudoeffective.
\end{proof}

Since $K_X+\omega$ is pseudoeffective, we may apply 
\cite[Thm.3.12]{Bou04} to obtain  a divisorial Zariski decomposition\footnote{The statements in \cite{Bou04} are for complex compact manifolds, 
but generalise immediately to our
situation: take $\holom{\mu}{X'}{X}$ a desingularisation, and let $m \in \N$ be the Cartier index of $K_X$. Then
$\mu^* (m(K_X+\omega))$ is a pseudoeffective class with divisorial Zariski decomposition 
$\mu^* (m(K_X+\omega))= \sum \eta_j S_j' + P'_\omega$. The decomposition of $K_X+\omega$ is defined by the push-forwards
$\mu_* (\frac{1}{m} \sum \eta_j S_j')$ and $\mu_* (\frac{1}{m} P'_\omega)$. Since a prime divisor $D \subset X$
is not contained in the singular locus of $X$, the decomposition has the stated properties.}
\begin{equation} \label{Bdecomposition}
K_X+\omega = \sum_{j=1}^r \lambda_j S_j + P_\omega,
\end{equation}
where the $S_j$ are integral surfaces in $X$, the coefficients $\lambda_j \in \R^+$ and $P_\omega$ is a 
pseudoeffective class which is nef in codimension one \cite[Prop.2.4]{Bou04}, that is for every surface $S \subset X$
the restriction $P_\omega|_S$ is pseudoeffective.

\begin{lemma} \label{lemmasurfaces}
Under the Assumption \ref{assumption}, let $S$ be a surface 
such that $(K_X+\omega)|_S$ is not pseudoeffective. Then 
$S$ is one of the surfaces $S_j$ in the divisorial Zariski decomposition \eqref{Bdecomposition} of $K_X+\omega$.
Moreover $S=S_j$ is Moishezon and any desingularisation $\hat S_j$ is a uniruled projective surface.
\end{lemma}

\begin{proof}
The proof that $S=S_j$ for some $j$ is analogous to the proof in \cite[Lemma 4.1]{HP15},
thus (up to renumbering) we may suppose that $S=S_1$.
We have
$$
S = S_1 = \frac{1}{\lambda_1} (K_X+\omega) - \frac{1}{\lambda_1} (\sum_{j=2}^r \lambda_j S_j +  P_\omega),
$$
so by adjunction
$$
K_{S} = (K_X+S)|_{S} = 
(\frac{\lambda_1+1}{\lambda_1} K_X|_S+\frac{1}{\lambda_1}\omega|_S) - 
\frac{1}{\lambda_1} (\sum_{j=2}^r \lambda_j (S_j \cap S) +  P_\omega|_S).
$$
Note now that $\frac{\lambda_1+1}{\lambda_1} K_X|_S+\frac{1}{\lambda_1}\omega|_S$ is not pseudoeffective:
otherwise
$$
\left(\frac{\lambda_1+1}{\lambda_1} K_X|_S+\frac{1}{\lambda_1}\omega|_S\right) + \omega|_S 
= 
\frac{\lambda_1+1}{\lambda_1} (K_X+\omega)|_S
$$
would be pseudoeffective, in contradiction to our assumption.
Since $$ \frac{1}{\lambda_1} (\sum_{j=2}^r \lambda_j (S_j \cap S) +  P_\omega|_S)$$ is pseudoeffective, the class $K_S$ cannot be pseudoeffective.

Let now
\holom{\pi}{\hat{S}}{S} be the composition of the normalisation and the minimal resolution of the surface $S$,
then by \eqref{equationresolve} there exists an effective divisor $E$ such that
$$
K_{\hat{S}} \sim_\Q \pi^* K_{S} - E. 
$$
Thus $K_{\hat{S}}$ is not pseudoeffective, in particular $\kappa(\hat S)=-\infty$. 
It follows from the Castelnuovo-Kodaira classification that $\hat{S}$ is covered by rational curves,
in particular $\hat S$ is a projective surface \cite{BHPV04}. Thus $S$ is Moishezon.
\end{proof}

\begin{corollary} \label{corollaryalgebraicallynef}
Under the Assumption \ref{assumption}, the adjoint class $K_X+\omega$
is nef if and only if
$$
(K_X+\omega) \cdot C \geq 0 
$$
for every curve $C \subset X$.
\end{corollary}

\begin{proof}
We prove the non-trivial implication by contradiction, so
suppose that $K_X+\omega$ is not nef, but
$(K_X+\omega) \cdot C \geq 0$ for all curves $C \subset X$. 
Since $K_X+\omega$ is pseudoeffective by Lemma \ref{lemmapseudoeffective} 
and the restriction to every curve is nef, there exists by  \cite{Pau98}, \cite[Prop.3.4]{Bou04}
an irreducible surface $S \subset X$ such that 
$(K_X+\omega)|_S$ is not pseudoeffective. 
Fix a  desingularisation \holom{\pi}{\hat S}{S} of the surface $S$.
By Lemma \ref{lemmasurfaces} the surface $\hat S$ is projective
and uniruled. The class $\pi^* (K_X+\omega)|_{S}$ is not pseudoeffective and, since $H^2(\hat S, \sO_{\hat S})=0$,
the class is represented by an $\R$-divisor. Thus there exists a covering family of curves $C_t \subset S$
such that 
$$
(K_X+\omega) \cdot C_t = \pi^* (K_X+\omega)|_{S} \cdot \hat C_t < 0,
$$
where $\hat C_t$ denotes the strict transform of $C_t$ in $\hat S.$ 
This contradicts our assumption that $(K_X+\omega) \cdot C \geq 0$ for all curves $C \subset X$. 
\end{proof}

\subsection{The adjoint cone theorem}

The goal of this subsection is to prove a cone theorem for the adjoint class $K_X+\omega$:

\begin{theorem} \label{theoremadjointNA}
Under the Assumption \ref{assumption}
there exists a countable family $(\Gamma_i)_{i \in I}$  of rational curves  on $X$
such that 
$$
0 < -(K_X+\omega) \cdot \Gamma_i \leq 4
$$
and
$$
\NAX = \NAX_{(K_X+\omega) \geq 0} + \sum_{i \in I} \R^+ [\Gamma_i] 
$$
\end{theorem}

The proof of Theorem \ref{theoremadjointNA} is quite similar to the proof of \cite[Thm.1.2.]{HP15}; for sakes of completeness
we explain the main steps:

\begin{lemma} \label{lemmabasic}
Under the Assumption \ref{assumption}, let
$C \subset X$ be a curve such that
$(K_X+\omega) \cdot C<0$ and $\dim_C \Chow{X}>0$.

Then there exists a unique surface $S_j$ from the divisorial Zariski decomposition \eqref{Bdecomposition}
such that $C$ and its deformations are contained in the surface $S_j$.
Moreover we have
\begin{equation} \label{canonicaldegree}
K_{S_j} \cdot C < K_X \cdot C.
\end{equation}
\end{lemma}

\begin{proof} Identical to the proof of \cite[Lemma 5.4]{HP15}, simply replace
$K_X$ by $K_X+\omega$.
\end{proof}

\begin{lemma} \label{lemmadeform}
Under the Assumption \ref{assumption}, let 
$S_1, \ldots, S_r$ be the surfaces appearing in the divisorial Zariski decomposition \eqref{Bdecomposition}.
Set
$$
b := \max \{ 1, -(K_X+\omega) \cdot Z \ | \ Z \ \mbox{a curve s.t.} \ Z \subset S_{j, \sing} \ \mbox{or} \ Z \subset S_{j} \cap S_{j'}, j \ne j'  \}.
$$
If $C \subset X$ is a curve
such that
$$
-(K_X+\omega) \cdot C > b,
$$
then we have $\dim_C \Chow{X}>0$.
\end{lemma}

In the proof we will use the following deformation property:

\begin{definition} \label{definitionveryrigid} \cite[Defn.4.3]{HP15}
Let $X$ be a  normal $\Q$-factorial K\"ahler threefold with at most terminal  singularities.
We say that a curve $C$ is very rigid if 
$$
\dim_{mC} \Chow{X} = 0
$$
for all $m>0$.
\end{definition}

\begin{proof}[Proof of Lemma \ref{lemmadeform}]
Since $\omega$ is nef, we have $-K_X \cdot C>b$.
The condition $b \geq 1$ implies that the curve $C$ is not very rigid (cf. \cite[Thm.4.5]{HP15}).
We can now argue exactly as in \cite[Lemma 5.6]{HP15} to deduce
$$
P_\omega  \cdot C \geq 0.
$$
Since $(K_X+\omega) \cdot C < 0$, the divisorial Zariski decomposition implies that there exists 
a number $j \in \{ 1, \ldots, r\}$ such that $S_j \cdot C<0$. In particular we have $C \subset S_j$. The class
$\omega$ being nef, we thus obtain 
$$
K_{S_j} \cdot C < K_X \cdot C < -b.
$$
By definition of $b$, the curve $C$ is not contained in the singular locus of $S_j$.
Let \holom{\pi_j}{\hat{S}_j}{S_j} be the composition of normalisation and minimal resolution (cf. Subsection \ref{subsectionsurfaces}). 
Then the strict transform
$\hat C$ of $C$ is well-defined and from  \eqref{inequalitycanonical} we deduce 
$$ K_{\hat S_j} \cdot \hat C \leq K_{S_j} \cdot C < -b.
$$
Since $b \geq 1$, \cite[Thm.1.15]{Kol96} yields
$$
\dim_{\hat C} \Chow{\hat S}>0,
$$
so $\hat C$ deforms. Thus its push-forward $\pi_* \hat C = C$ deforms.
\end{proof}

\begin{corollary} \label{corollarybreak}
Under the Assumption \ref{assumption},
let $b$ be the constant from Lemma \ref{lemmadeform} and set
$$
d := \max \{ 3, b \}.
$$
If $C \subset X$ is a curve
such that $-(K_X+\omega) \cdot C > d$,
we have
$$
[C] = [C_1] + [C_2]
$$
with $C_1$ and $C_2$ effective 1-cycles (with integer coefficients) on $X$.
\end{corollary}

\begin{proof}
Since $\omega$ is nef, we have $-K_X \cdot C>d$. Using the Lemmas \ref{lemmabasic} and \ref{lemmadeform},
the proof of \cite[Cor.5.7]{HP15} applies without changes. 
\end{proof}

\begin{lemma} \label{lemmarationalrepresentative}
Under the Assumption \ref{assumption}, let $\R^+ [\Gamma_i]$ be a $(K_X+\omega)$-negative extremal ray in $\NEX$, where
$\Gamma_i$ is a curve that is not very rigid (cf. Definition \ref{definitionveryrigid}). Then the following holds:
\begin{enumerate}
\item There exists a curve $C \subset X$ such that $[C] \in \R^+ [\Gamma_i]$ and $\dim_C \Chow{X}>0$.
\item There exists a rational curve $C \subset X$ such that $[C] \in \R^+ [\Gamma_i]$.
\end{enumerate}
\end{lemma}

\begin{proof}
This is completely analogous to  \cite[Lemma 5.8]{HP15} since the existence of the rational
curve $C \subset X$ such that $[C] \in \R^+ [\Gamma_i]$ is a consequence of \cite[Lemma 5.5 a)]{HP15}
which contains no assumption on $K_X$.
\end{proof}

Following the strategy of \cite[Thm.6.2]{HP15} we first establish the cone theorem for $K_X + \omega.$ 

\begin{theorem} \label{theoremNE}
Under the Assumption \ref{assumption},
there exists a number $d \in \N$ and a countable family $(\Gamma_i)_{i \in I}$  of curves  on $X$
such that 
$$
0 < -(K_X+\omega) \cdot \Gamma_i \leq d
$$
and
$$
\NEX = \NEX_{(K_X+\omega) \geq 0} + \sum_{i \in I} \R^+ [\Gamma_i].
$$
If the ray $\R^+ [\Gamma_i]$ is extremal in $\NEX$, there exists a rational curve $C_i$ on $X$ 
such that $[C_i] \in \R^+ [\Gamma_i]$.
\end{theorem}

\begin{proof}
Let $d \in \N$ be the bound from Corollary \ref{corollarybreak}. 
There are only countably many curve classes $[C] \in \NEX$,
such that  $$0 < -(K_X+\omega) \cdot C \leq d.$$
We choose a representative $\Gamma_i$ for each such class $[C]$ and
set
$$
V := \NEX_{(K_X+\omega) \geq 0} + \sum_{0 < -(K_X+\omega) \cdot \Gamma_i \leq d} \R^+ [\Gamma_i].
$$
Fix a K\"ahler class $\eta$ on $X$ such that 
$
\eta \cdot C \geq 1
$ 
for every curve $C \subset X$

{\em Step 1. We have $\NEX=V$.} 
By \cite[Lemma 6.1]{HP15} it is sufficient to prove that $\NEX = \overline V$, i.e.
the class $[C]$ of every irreducible curve $C \subset X$
is contained in $V$. 
We will prove the statement by induction on the degree $l:=\eta \cdot C$. The start
of the induction for $l=0$ is trivial.
Suppose now that we have shown the statement for all curves of degree at most $l-1$
and let $C$ be a curve such that $l-1 < \eta \cdot C \leq l$.
If $-(K_X+\omega) \cdot C \leq d$ we have $[C] \in V$ by definition. Otherwise there exists by Corollary \ref{corollarybreak} 
a decomposition
$$
[C] = [C_1] + [C_2]
$$
with $C_1$ and $C_2$ effective 1-cycles (with integer coefficients) on $X$.
Since $\eta \cdot C_i \geq 1$ for $i=1,2$ we have
$\eta \cdot C_i \leq l-1$ for $i=1,2$. By induction both classes are in $V$, so $[C]$ is in $V$.

{\em Step 2. Every extremal ray contains the class of a rational curve.}
If the ray $\R^+ [\Gamma_i]$ is extremal in $\NEX$, we know by \cite[Thm.4.5]{HP15}
and Lemma \ref{lemmarationalrepresentative} that there exists a rational curve $C_i$ such that $[C_i]$
is in the extremal ray.
\end{proof}

We next pass from $\NEX$ to $\NAX:$ 

\begin{theorem} \label{theoremNApreliminary}
Under the Assumption \ref{assumption}
there exists a number $d \in \N$ and a countable family $(\Gamma_i)_{i \in I}$  of curves  on $X$
such that 
$$
0 < -(K_X+\omega) \cdot \Gamma_i \leq d
$$
and
$$
\NAX = \NAX_{(K_X+\omega) \geq 0} + \sum_{i \in I} \R^+ [\Gamma_i]. 
$$
If the ray $\R^+ [\Gamma_i]$ is extremal in $\NAX$, there exists a rational curve $C_i$ on $X$ 
such that $[C_i] \in \R^+ [\Gamma_i]$.
\end{theorem}

Theorem \ref{theoremNApreliminary} is a consequence of Theorem \ref{theoremNE} and the following proposition.

\begin{proposition} \label{propositionNA}
Under the Assumption \ref{assumption}, 
suppose that there exists a $d \in \N$ and a countable family $(\Gamma_i)_{i \in I}$  of curves  on $X$
such that 
$$
0 < -(K_X+\omega) \cdot \Gamma_i \leq d
$$
and
$$
\NEX = \NEX_{(K_X+\omega) \geq 0} + \sum_{i \in I} \R^+ [\Gamma_i].
$$
Then we have
$$
\NAX = \NAX_{(K_X+\omega) \geq 0} + \sum_{i \in I} \R^+ [\Gamma_i].
$$
\end{proposition}

\begin{proof} Identical to the proof of \cite[Prop.6.4]{HP15}: simply replace $K_X$ by $K_X+\omega$
and note that the uniruledness of a surface $S \subset X$ such that 
$(K_X+\omega)|_S$ is not pseudoeffective is proven in Lemma \ref{lemmasurfaces}.
\end{proof}

Finally, Theorem \ref{theoremadjointNA} follows from Theorem \ref{theoremNApreliminary} in the same way as \cite[Thm.1.2]{HP15}
is deduced from \cite[Thm.6.3]{HP15}. 

\subsection{The adjoint contraction theorem}

In this subsection we prove the contraction theorem:

\begin{theorem} \label{theoremadjointcontraction}
Under the Assumption \ref{assumption}, 
let $\R^+ [\Gamma_i]$ be a $(K_X+\omega)$-negative extremal ray in $\NAX$.
Then the contraction of $\R^+ [\Gamma_i]$ exists in the K\"ahler category.
\end{theorem}

For the rest of this subsection we 
fix  $R := \R^+ [\Gamma_{i_0}]$ a $(K_X+\omega)$-negative extremal ray in $\NAX$.

\begin{definition}
We say that the $(K_X+\omega)$-negative extremal ray $R$ is small if every curve $C \subset X$ with
$[C] \in R$ is very rigid in the sense of Definition \ref{definitionveryrigid}.
Otherwise we say that the extremal ray $R$ is divisorial.
\end{definition}

\begin{remark} \label{remarkrays}
Notice that, due to Assumption \ref{assumption} and Lemma \ref{lemmapseudoeffective},  through a general point $x \in X$ there is no curve $C$ belonging to $R$. Hence 
the curves belonging to $R$ cover at most a divisor. 

If the extremal ray $R$ is small, standard arguments
show that there are only finitely many curves
$C \subset X$ such that $[C] \in R$ (cf. \cite[Rem.7.2]{HP15}). 

If the extremal ray $R$ is divisorial, we can argue as in \cite[Lemma 7.5]{HP15}
that there exists a  unique surface $S \subset X$ such that 
$$
S \cdot R < 0.
$$
In particular any curve $C \subset X$ with $[C] \in R$ is contained in $S$.
\end{remark}

The following proposition is a well-known consequence of the cone theorem \ref{theoremNApreliminary},
cf. \cite[Prop.7.3]{HP15} for details:

\begin{proposition} \label{propositionperp} 
There exists a nef class $\alpha \in N^1(X)$ such that
$$ 
R = \{ z \in \overline{NA}(X) \ \vert \ \alpha \cdot z = 0\},
$$
and such that, using the notation of Theorem \ref{theoremNApreliminary}, the class $\alpha$ is strictly positive on
$$
\left( \NAX_{(K_X+\omega) \geq 0} + \sum_{i \in I, i \neq i_0 } \R^+ [\Gamma_i] \right) \setminus \{0\}.
$$ 
We call $\alpha$ a nef supporting class for the extremal ray $R =  \R^+ [\Gamma_{i_0}].$
\end{proposition} 

In what follows we will use at several places the following theorem, stated in \cite[Thm.2.6]{8authors}
for projective manifolds:

\begin{theorem} \label{theoremnefreduction}
Let $X$ be a normal compact K\"ahler space, and let $\alpha$ be a nef cohomology 
class on $X$. Then there exists an almost holomorphic, dominant
meromorphic map $f: X \dashrightarrow Y$ with connected fibers, such that
\begin{enumerate} 
\item $\alpha$ is numerically trivial on all compact fibers $F$ of
    $f$ with $\dim F = \dim X - \dim Y$
  \item for every general point $x \in X$ and every irreducible curve
    $C$ passing through $x$ with $\dim f(C) > 0$, we have $\alpha
    \cdot C > 0$.
  \end{enumerate}
In particular, if two general points of $X$ can be joined by a chain
$C$ of curves such that $\alpha \cdot C = 0,$ then $\alpha \equiv 0$.
\end{theorem}

For the convenience of the reader we sketch how to adapt the proof from \cite{8authors} to this setting.

\begin{proof} We define that two points $x, y \in X$ are equivalent if they can be joined by a
connected curve $C$ such that $\alpha \cdot C = 0$. By \cite[Thm.1.1]{Cam04b} 
there exists an almost holomorphic map $f: X \dashrightarrow Y$ with connected fibers to a normal compact K\"ahler space $Y$
such that two general points $x$ and $y$ are equivalent if and
only if $f(x) = f(y)$. By construction a general $f$-fibre $F_0$ is a normal compact K\"ahler space such that two general points
can be connected by a curve, thus $F_0$ is algebraic \cite[p.212, Cor.]{Cam81}. Hence we can apply \cite[Thm.2.4]{8authors}
to see that $\alpha|_{F_0}=0$. In particular for any K\"ahler form $\omega$ on $X$ we have $\alpha \cdot \omega^{d-1} \cdot F_0=0$
where $d:=\dim X-\dim Y$. Since any compact $f$-fibre $F$ of dimension $d$ is homologous to some multiple of $F_0$ and $\alpha$ is nef we see that $\alpha|_F=0$.
\end{proof}

\begin{notation} {\rm 
Suppose that the extremal ray $R = \R^+ [\Gamma_{i_0}]$ is divisorial, and let $S$ be the surface 
such that $S \cdot R<0$ (cf. Remark \ref{remarkrays}).
Let $\nu: \tilde S \to S \subset X$ be the normalisation.
By Lemma \ref{lemmarationalrepresentative}(a) there exists a curve $C \subset X$ such that $[C] \in R$ and $\dim_C \Chow{X}>0$.
Since we have $S \cdot C<0$, the
deformations  $(C_t)_{t \in T}$ of $C$ induce a
dominating family $(\tilde C_t)_{t \in T'}$ of $\tilde S$ such that $\nu^*(\alpha) \cdot \tilde C_t=0$. The class
$\nu^*(\alpha)$ is a nef class on $\tilde S$ and we may consider the nef reduction
$$ 
\tilde f: \tilde S \to \tilde B$$
with respect to $\nu^*(\alpha)$, 
cf. Theorem \ref{theoremnefreduction}.
By definition of the nef reduction
this implies
$$ n(\alpha) := \dim \tilde B \in \{0,1\}.$$}
\end{notation} 

\begin{lemma} \label{lemmacontractdivisor}
\begin{enumerate} 
\item Suppose that the extremal ray $R$ is divisorial and $n(\alpha) = 0$. Then the surface 
$S$ can be blown down to a point $p$: there exists a bimeromorphic
morphism $\varphi: X \to Y$ 
to a normal compact  threefold $Y$ with $\dim \varphi(S) = 0$ such that $\varphi|_{X \setminus S}$ is an isomorphism onto $Y \setminus \{p\}$.
\item Suppose that the extremal ray $R$ is divisorial and $n(\alpha) = 1$. 
Then there exists a fibration $\holom{f}{S}{B}$ onto a curve $B$ such that
a curve $C \subset S$ is contracted if and only if $[C] \in R$.
Moreover the surface $S$ can be contracted onto a curve:
there exists a bimeromorphic
morphism $\varphi: X \to Y$ 
to a normal compact threefold $Y$ such that
$\varphi|_S=f$ and $\varphi|_{X \setminus S}$ is an isomorphism onto $Y \setminus B$.
\end{enumerate} 
\end{lemma}

\begin{proof}
The proof is identical to the proofs of \cite[Cor.7.7, Lemma 7.8, Cor.7.9]{HP15} which only use properties of the
nef class $\alpha$ and $K_X \cdot R<0$ which holds
since $\omega \cdot R>0$.
\end{proof}

\begin{notation} \label{notationsmalllocus}{\rm
Suppose that the extremal ray $R = \R^+ [\Gamma_{i_0}]$ is small.
Set 
$$
C := \cup_{C_l \subset X, [C_l] \in R} C_l,
$$
then $C$ is a finite union of curves by Remark \ref{remarkrays}. We say that $C$ is contractible
if there exists a bimeromorphic
morphism $\varphi: X \to Y$ 
onto a normal threefold $Y$ with $\dim \varphi(C) = 0$ such that $\varphi|_{X \setminus C}$ is an isomorphism onto 
$Y \setminus \varphi(C)$.}
\end{notation}

The following statement is a variant
of \cite[Prop.7.11]{HP15}.

\begin{proposition} \label{propositionalphabig}
Suppose that the extremal ray $R=\R^+ [\Gamma_i]$ is small. 
Let $S \subset X$ be an irreducible surface. Then we have
$\alpha^2 \cdot S>0$.
\end{proposition}

\begin{proof}
By hypothesis, the cohomology class $\alpha-(K_X+\omega)$ is positive on the extremal ray $R$, moreover we know 
by Proposition \ref{propositionperp} that $\alpha$ is positive on
$$
\left( \NAX_{(K_X+\omega) \geq 0} + \sum_{i \in I, i \neq i_0 } \R^+ [\Gamma_i] \right) \setminus \{0\}.
$$ 
Thus, up to replacing $\alpha$ by some positive multiple, we may suppose that $\alpha-(K_X+\omega)$ is
positive on $\NAX \setminus \{ 0 \}$. Since $X$ is a K\"ahler space, this implies by \cite[Cor.3.16]{HP15} 
that 
$$
\eta := \alpha-(K_X+\omega)
$$ 
is a K\"ahler class.
Arguing by contradiction we suppose
that $\alpha^2 \cdot S=0$.

We first claim that $(K_X+\omega)|_S$ is not pseudoeffective.
If $\alpha|_S=0$ this is obvious, so suppose $\alpha|_S \neq 0$.
Then we have
$$
0 = \alpha^2 \cdot S = (K_X+\omega) \cdot \alpha \cdot S
+ \eta \cdot \alpha \cdot S 
$$ 
and 
$$
\eta \cdot \alpha \cdot S = \eta|_S \cdot \alpha|_S >0
$$ 
by the Hodge index theorem (note hat if $\holom{\pi}{S'}{S}$ is a desingularisation, then $\pi^*(\eta|_S)$ is nef and big and $\pi^*(\alpha|_S)$
is nef, so the ``smooth'' Hodge index theorem applies). Thus we have
\begin{equation} \label{help1}
(K_X+\omega) \cdot \alpha \cdot S = (K_X+\omega)|_S \cdot \alpha|_S < 0.
\end{equation}
In particular $(K_X+\omega)|_S$ is not pseudoeffective, 
the class $\alpha|_S$ being nef. 

Since $(K_X+\omega)|_S$ is not pseudoeffective, we know by
Lemma \ref{lemmasurfaces} that $S$ is uniruled and  one of
the surfaces in the Zariski decomposition \eqref{Bdecomposition}. In particular we cannot have
$\alpha|_S = 0$ since $S$ contains infinitely many curves (recall that the ray $R$ is small, hence $\alpha \cdot C = 0$ can occur only for
finitely many curves $C$). 
Using the decomposition \eqref{Bdecomposition} and \eqref{help1} we obtain $\alpha \cdot S^2<0$, hence
\begin{equation} \label{help2}
(K_X+\omega+S) \cdot \alpha \cdot S < 0.
\end{equation}
Let \holom{\pi}{\hat S}{S} be the composition of the normalisation and the minimal resolution (cf. Subsection \ref{subsectionsurfaces}), then \eqref{equationresolve}
and \eqref{help2} imply that
\begin{equation} \label{help3}
(K_{\hat S}+ \pi^* \omega|_S) \cdot \pi^* \alpha|_S < 0.
\end{equation}
Since the surface $\hat S$ is projective, the
nef class $\pi^* \alpha|_S$ is represented by
an $\R$-divisor. 
The extremal ray $R$ contains only the classes of finitely many curves, so $\pi^* \alpha$ is strictly
positive on every movable curve in $\hat S$.

Fix an ample divisor $A$ on $\hat S$. 
By \cite[Thm.1.3]{Ara10} for every $\varepsilon>0$ we have a decomposition
$$
\pi^* \alpha|_S = C_{\varepsilon} + \sum \lambda_{i, \varepsilon} M_{i, \varepsilon}
$$
where $\lambda_{i, \varepsilon} \geq 0$,  the $M_{i, \varepsilon}$ are movable curves and $(K_{\hat S}+\varepsilon A) \cdot C_\varepsilon \geq 0$. The class $\pi^* \alpha|_S$ is strictly
positive on every movable curve in $\hat S$, so we have $\pi^* \alpha|_S \cdot M_{i, \varepsilon}>0$.
Since $(\pi^* \alpha|_S)^2=0$ and $\pi^* \alpha|_S \cdot M_{i, \varepsilon}>0$ we 
must have
$\pi^* \alpha|_S = C_{\varepsilon}$
for all $\varepsilon>0$. Passing to the limit we obtain $K_{\hat S} \cdot \pi^* \alpha|_S \geq 0$, a contradiction to \eqref{help3}.
\end{proof}

\begin{theorem} \label{theoremsmallray}
Suppose that the extremal ray $R$ is small. 
Then $C$ is contractible. 
\end{theorem}

\begin{proof} 
Let $\alpha  \in N^1(X)$ be the nef class supporting $R$ as in Proposition~\ref{propositionperp}. 
{\it We claim that the class $\alpha$ is big}, i.e., if $\holom{\pi}{X'}{X}$ is a desingularisation then the pull-back $\pi^* \alpha$ is a big cohomology class. Once we have shown this property, the proof of \cite[Thm.7.12]{HP15} applies.

{\em Proof of the claim.}
By definition of the class $\alpha$, the class  $-(K_X+\omega) +\alpha$ is positive on the extremal ray $R$.
Since $\alpha$ is strictly positive on 
$$
\left( \NAX_{(K_X+\omega) \geq 0} + \sum_{i \in I, i \neq i_0 } \R^+ [\Gamma_i] \right) \setminus \{0\},
$$ 
we may suppose, up to replacing $\alpha$ by some positive multiple, that $-(K_X+\omega) +\alpha$ 
is strictly positive on this cone. In total, $-(K_X+\omega) + \alpha$ is strictly positive on $\NAX \setminus \{ 0 \}$.
Thus $-(K_X+\omega) + \alpha$ is a K\"ahler class 
by \cite[Cor.3.16]{HP15}, i.e., we may write
$$
\alpha = (K_X+\omega) + \eta,
$$
where $\eta$ is a K\"ahler class. We know by Lemma \ref{lemmapseudoeffective} that $K_X+\omega$ is pseudoeffective.
Thus $\pi^* \alpha$ is the sum of the pseudoeffective class $\pi^* (K_X+\omega)$ and the nef and big class $\pi^* \eta$,
hence it is big.
\end{proof} 

\begin{proof}[Proof of Theorem \ref{theoremadjointcontraction}]
The existence of a morphism \holom{\varphi}{X}{Y} contracting exactly the curves in the extremal ray 
is established in Lemma  \ref{lemmacontractdivisor}
and in Theorem \ref{theoremsmallray}. Since $\omega$ is nef,
the extremal ray $\R^+ [\Gamma_i]$ is $K_X$-negative.
Therefore, applying  \cite[Cor.8.2]{HP15}, it follows that
$Y$ is a K\"ahler space.
\end{proof}

\subsection{Proof of Theorem \ref{theoremadjointMMP}}

\begin{proof}

{\em Step 1: Running the MMP.} If $K_X+\omega$ is nef
for every normalised K\"ahler class $\omega$, we are finished. Suppose therefore that $K_X+\omega$ is not nef.
Then there exists by Theorem \ref{theoremadjointNA} a $(K_X+\omega)$-negative extremal ray $R$ in $\NAX$.
By Theorem \ref{theoremadjointcontraction} the contraction  \holom{\varphi}{X}{Y} of $R$ exists in the K\"ahler category. 
Note that since $\omega$ is nef, the canonical class $K_X$
is negative on the extremal ray $R$.

If $R$ is divisorial we can continue the MMP with $Y$ by \cite[Prop.8.1.c)]{HP15}.  
If $R$ is small, we know by Mori's flip theorem \cite[Thm.0.4.1]{Mor88} that the flip $\holom{\varphi^+}{X^+}{Y}$ exists, and
by \cite[Prop.8.1.d)]{HP15} we can continue the MMP with $X^+$ (which is again K\"ahler). 

{\em Step 2: Termination of the MMP.} Recall that for a normal compact threefold $X$ with at most terminal singularities, 
the difficulty $d(X)$ \cite{Sho85} is defined by
$$
d(X) := \# \{ i \ | \ a_i < 1 \},
$$
where $K_Y = \mu^* K_X + \sum a_i E_i$ and $\holom{\mu}{Y}{X}$ is any resolution of singularities.
Recall that any contraction in our MMP is a $K_X$-negative
contraction, so
by \cite[Lemma 5.1.16]{KMM87}\footnote{The proof is local in a neighbourhood of the flipping locus, so it holds without change in the analytic setting.}, \cite{Sho85} we have $d(X)>d(X^+)$, if $X^+$ is the flip of a small contraction.
Since the Picard number and the difficulty are non-negative integers, any MMP terminates after finitely many steps.
\end{proof}


\section{The base-point free theorem}

We first prove Theorem \ref{theorembpf}, which is the analogue of the base point free theorem in the non-algebraic case.

\subsection{Proof of Theorem \ref{theorembpf}}

\begin{proof}
We will use the nef reduction of $X$ with respect to the cohomology class $K_X + \omega$,
cf. Theorem \ref{theoremnefreduction}. 
We denote by $n(K_X + \omega)$ the dimension of the base of the nef reduction of $K_X+\omega$ and claim that
$$  n(K_X + \omega) = 2. $$
Notice first that the general fibres of the MRC-fibration provide
a dominating family of curves which is $K_X+\omega$-trivial, so $  n(K_X + \omega)  \leq 2.$\\
If $n(K_X +\omega) = 1$ the nef reduction is a holomorphic fibration $X \rightarrow C$ (cf. \cite[2.4.3]{8authors}) 
and $K_X+\omega$ is numerically trivial on the general fibre by Theorem \ref{theoremnefreduction}.
In particular the general fiber is a smooth Fano surface, hence rationally connected,
a contradiction to our assumption on the base of the MRC-fibration. \\
If  $n(K_X + \omega) = 0,$ then $K_X + \omega \equiv 0,$ hence $X$ is Fano
and rationally connected, again a contradiction.

Let $Z$ be a resolution of singularities of the unique irreducible component of $\Chow{X}$ such that the general
point corresponds to the general fibre of the MRC-fibration. Let $\Gamma$ be the normalisation of the pull-back of
the universal family and denote by $\holom{p}{\Gamma}{X}$ and $\holom{q}{\Gamma}{Z}$ the natural morphisms.
Since $\Gamma$ is in Fujiki's class $\mathcal C$, the
surface $Z$ is in the class $\mathcal C$  by \cite[Thm. 3]{Var86}. 
A smooth surface in the class $\mathcal C$ is K\"ahler,
so $Z$ is K\"ahler.

We claim that there exists a big and nef  class $\alpha$ on $Z$
such that 
$$
p^* (K_X+\omega) = q^* \alpha.
$$

{\em Step 1. Construction of the class $\alpha$.}
Set $\Gamma_z = q^{-1}(z) $ for $z \in Z.$ 
Note first that we have $R^1 q_*(\sO_{\Gamma}) = 0$ (the morphism $q$ is projective, so we can apply \cite[II, 2.8.6.2]{Kol96}). 
Using the exponential sequence this implies  $R^1 q_*(\mathbb Z) = 0$ and hence $R^1 q_*(\mathbb R) = 0$ by the universal coefficient theorem. 
Now we apply the Leray spectral sequence for $q$ and the sheaf $\mathbb R$. By what precedes we have 
$$E_2^{0,1} = H^0(Z,R^1q_*(\mathbb R)) = 0$$
and $$E_2^{1,1} = H^1(Z,R^1q_*(\mathbb R)) = 0.$$
Therefore 
$E_2^{2,0} = H^2(Z,\mathbb R)$ 
embeds into $H^2(\Gamma, \mathbb R)$, 
and
it suffices to show that the section $s \in E_2^{0,2} = H^0(Z,R^2 q_*(\mathbb R))$ which is given by
the class 
$$
s(z) = [p^*(K_X + \omega) \vert \Gamma_z] \in H^2(\Gamma_z,\mathbb R),
$$
vanishes for every $z \in Z$. 
By definition of a normalised K\"ahler class we
have $s(z)=0$ for $z \in Z$ general. Since $p^*(K_X + \omega)$ is nef, this implies 
that the class $p^*(K_X + \omega)$ is zero on all the irreducible components of any fibre $\Gamma_z$.
Thus we have $s(z) = 0$ for all $z \in Z$
proving the existence of $\alpha$. Note that since $q^* \alpha$ is nef, the class $\alpha$ is nef \cite[Thm.1]{Pau98}.

{\it Step 2. Intersection numbers.} 
Let $D \subset \Gamma$ be an irreducible component (possibly of dimension $1$)
of the $p$-exceptional locus. 
Since $p$ is finite on the
fibres of $q$, there exists a curve $C \subset D$ that is
contracted by $p$ and such that $q(C)$ is not a point. In
particular we have
$$
\alpha \cdot q(C) = q^* \alpha \cdot C =
p^* (K_X+\omega) \cdot C = 0.
$$
Since the meromorphic map $X \dashrightarrow Z$
is almost holomorphic, $D$ does not surject
onto $Z$. Thus we have $q(D)=q(C)$, and by what precedes
we obtain
$$
(q^* \alpha)|_D = 0.
$$
Note now that, $\Gamma$ being a modification of 
a threefold which has a finite singular locus, the singular locus of $\Gamma$ is a union
of curves which are contained in the $p$-exceptional
locus and finitely many points. 
Let $\holom{\mu}{\hat X}{\Gamma}$ be a desingularisation
such the exceptional set of $\hat p:= p \circ \mu$ has pure codimension one. Set moreover $\hat q:=q \circ \mu$.
By what precedes, 
\begin{equation} \label{vanish}
\hat q^* \alpha \cdot {\hat D} = 0   \ \mbox{in $N_1(\hat X)$} 
\end{equation}
for every irreducible component $\hat D$ of the $\hat p$-exceptional locus. 

{\em Step 3. The class $\alpha$ is big, i.e. we have $\alpha^2>0$.}
Since $\omega$ is a K\"ahler class, we know that, up to replacing
$\hat X$ by some further blowup, 
there exists an effective $\Q$-divisor $F$ with support
in the $\hat p$-exceptional locus such that
$$
\hat p^* \omega - F
$$
is a K\"ahler class. Being a K\"ahler class is an open property, so
there exists a K\"ahler class $\eta_Z$ 
on $Z$ such that
$$
\hat p^* \omega - F - \hat q^* \eta_Z
$$
is a K\"ahler class. Using P\v aun's theorem \cite[Thm.1.1]{Pau12}  
as in the proof of Lemma \ref{lemmapseudoeffective}, we conclude that
$$
K_{\hat X/Z} + \hat p^* \omega - F - \hat q^* \eta_Z
$$
is pseudoeffective. Since $X$ has terminal singularities, 
$$
K_{\hat X} = \hat p^* K_X + E
$$ 
with $E$ an effective $\Q$-divisor 
supported on the $\hat p$-exceptional locus.
Consider now the decomposition
\begin{equation} \label{decompose}
\hat p^* (K_X+\omega) = 
[K_{\hat X/Z}+ \hat p^* \omega - F - \hat q^* \eta_Z]
- E + F + \hat q^* K_Z + \hat q^* \eta_Z.
\end{equation}
We are going to intersect this equation with $\hat q^*(\alpha)$ in order to compute 
$$
\hat q^* \alpha^2 
= \hat q^* \alpha \cdot \hat p^* (K_X+\omega).
$$
Since $\alpha$ is nef, the intersection product
$$
\hat q^* \alpha \cdot  [K_{\hat X/Z}+ \hat p^* \omega - F - \hat q^* \eta_Z]
$$
is an element of $\NA{\hat X}$.
By \eqref{vanish} we have $\hat q^* \alpha \cdot (-E+F)=0$.
The surface $Z$ is not uniruled since it is the base of the MRC-fibration (cf. Remark \ref{remarkassumption}). Thus $K_Z$ is pseudoeffective, in particular the intersection product
$\hat q^* \alpha \cdot \hat q^* K_Z$
is an element of $\NA{\hat X}$.
Recall now that $\alpha \neq 0$ since $K_X+\omega \neq 0$.
Since $\eta_Z$ is a K\"ahler class and $\alpha$ is a non-zero nef class, the Hodge index
theorem yields $\eta_Z \cdot \alpha>0$. Thus 
$$
q^* \alpha \cdot q^* \eta_Z
$$
is a non-zero element of $\NA{\hat X}$. In total we obtain that
$$
\hat q^* \alpha^2 
= \hat q^* \alpha \cdot \hat p^* (K_X+\omega)
$$
is a non-zero element of $\NA{\hat X}$. Thus we have
$\alpha^2 \neq 0$.
 
{\em Step 4. Construction of the fibration $\varphi$.}
Let 
$$
E := \cup E_j \subset Z
$$
be the union of curves $E_j \subset Z$ such that $\alpha \cdot E_j=0$. Since $\alpha$ is nef and big,
the Hodge index theorem implies that the intersection form
on $E$ is negative definite. In particular $E$ is a finite set. By Grauert's criterion
there exists a bimeromorphic
morphism $\holom{\nu}{Z}{S}$
such that $E$ equals the $\nu$-exceptional locus. 
Since $Z$ is a K\"ahler surface and $\nu$ contracts
only subvarieties onto points, the surface $S$ is K\"ahler. In fact, take any K\"ahler form $\omega$ on $Z.$ Then  the class of the K\"ahler current 
$\nu_*(\omega) $  contains a K\"ahler form by \cite[Prop.3.3(iii)]{DP04}.

We claim that the fibration $\holom{\nu \circ q}{\Gamma}{S}$ factors through the bimeromorphic map $p$, i.e., there exists a holomorphic
fibration $\holom{\varphi}{X}{S}$ such that $\nu \circ q = \varphi \circ p$. By the rigidity lemma \cite[Lemma 4.1.13]{BS95}  it is sufficient to prove that every $p$-fibre
is contracted by $\nu \circ q$. Since $p$ is a Moishezon morphism, it moreover suffices to show that every curve $C \subset \Gamma$ such that
$p(C)$ is a point is contracted by $\nu \circ q$. Yet for such a curve $C$ we have
$$
q^* \alpha \cdot C = p^*(K_X+\omega) \cdot C = 0.
$$
It follows that $q(C) \subset E$, hence $q(C)$ is a point. This shows the existence of the fibration $\varphi$; by construction
the class $K_X+\omega$ is $\varphi$-trivial.
\end{proof}

\subsection{MMP for uniruled K\"ahler threefolds}

Recall that in our context
a normal K\"ahler space $X$ is $\Q$-factorial
if every Weil divisor $D \subset X$ is $\Q$-Cartier
and some reflexive power $\omega_X^{[m]}$ of the
dualising sheaf $\omega_X$ is locally free.

\begin{lemma} \label{lemmacontractionfibretype}
Let $X$ be a normal $\Q$-factorial compact K\"ahler threefold with at most terminal singularities.
Let $\holom{\varphi}{X}{S}$ be an elementary
Mori contraction onto a normal compact
surface, i.e., $\rho(X/S) = 1$ and $-K_X$ is $\varphi$-ample. 

Then $S$ is $\Q$-factorial and has at most klt singularities.
\end{lemma}

\begin{remark} \label{remarkconicbundle}
In the situation above, the fibration $\varphi$ is equidimensional since an elementary contraction
of fibre type does not contract a divisor. 
For a point $s \in S$ denote by $X_s$ the fibre over $s$, 
and let $ A \subset S$ be the set of all $s$ such that the fiber $X_s$ is singular at some point $x_0$ and such that 
$X$ is not smooth at $x_0.$ Then $A$ is finite, set $S_0 = S \setminus A$ and $X_0 = X \setminus \varphi^{-1}(A).$ 
The fiber space $f_0: X_0 \to S_0$ is a conic bundle. The sheaf $f_*(\omega_{X/S}) $ is reflexive, but might have singularities 
on $A$, so that $f$ might globally not be a conic bundle. However, $H^1(X_s,\mathcal O_{X_s}) = 0,$ in particular,
every irreducible component of any fiber $X_s$ is a smooth rational curve. 
\end{remark} 

\begin{proof}[Proof of Lemma \ref{lemmacontractionfibretype}]
Arguing as in \cite[5-1-5]{KMM87}, 
every Weil divisor $D \subset S$ is $\Q$-Cartier.

In order to see that $S$ has at most klt singularities we proceed as in the algebraic case. 
The claim is local on the base $S$, so given a point $0 \in S_{\sing}$ we fix a small analytic neighbourhood $0 \in U \subset S$. Since $X$ is smooth in codimension two and, by Remark \ref{remarkconicbundle} the projective morphism 
$\varphi$ is a conic bundle in the complement of the fibre $X_0$, 
there exists a smooth analytic subvariety $H \subset \fibre{\varphi}{U}$ such that
$H \rightarrow U$ is finite and \'etale in codimension one. By \cite[Prop.5.20]{KM98}
the surface $U$ has at most klt singularities. In particular some reflexive power $\omega_X^{[m]}$ of the
dualising sheaf $\omega_X$ is locally free.
\end{proof}

\begin{proof}[Proof of Theorem \ref{theoremMFS}]
By Theorem \ref{theoremadjointMMP}, there exists a MMP $X \dasharrow X'$ such that $K_{X'} + \omega'$ is nef for all normalised K\"ahler classes
$\omega' $ on $X'$. Fix such a K\"ahler class $\omega'$. 
Then apply the base point free theorem \ref{theorembpf} to the variety $X'$ to obtain a fibration
$\holom{\varphi}{X'}{S'}$ onto a surface $S'$ such that
$-K_{X'}$ is $\varphi$-ample.  In particular $\varphi$ is a projective morphism. Thus we can run the MMP of $X'$ over $S'$ using
the relative version of the cone and contraction theorem as in \cite[Sect.4]{Nak87}, \cite[Sect.3.6]{KM98}. 
As in the proof of Theorem \ref{theoremadjointMMP}
we can use the Picard number $\rho(X')$ and the difficulty $d(X')$ to
show that the MMP terminates. Since $K_{X'}$ is
not pseudoeffective over $S'$, the outcome of the MMP
$$
X' \dashrightarrow X''
$$
is a Mori fibre space $X'' \rightarrow S''$ over $S',$ 
with $S''$
a normal compact complex surface that dominates $S'$.
Since $S'$ is K\"ahler, and the bimeromorphic morphism $S'' \rightarrow S'$ is projective (we can always find an anti-effective exceptional divisor that is relatively ample), the surface $S''$ is K\"ahler.  
The properties of $S''$ are proven in Lemma \ref{lemmacontractionfibretype}.
\end{proof}


\newcommand{\etalchar}[1]{$^{#1}$}
\def\cprime{$'$}

\end{document}